\documentclass[12pt,a4paper]{article}
\usepackage[anythingbreaks]{breakurl}
\usepackage[breaklinks,pdfencoding=unicode]{hyperref}
\usepackage[british]{babel}
\usepackage[UKenglish]{isodate}
\usepackage{amssymb, amsmath, calligra, mathrsfs, color, fancyhdr, enumitem, amsthm, graphicx, array, tikz-cd, subfiles}
\usepackage[toc,page]{appendix}
\usepackage[a4paper, top=3.5cm, bottom=3cm, left=3.0cm, right=3.0cm]{geometry}
\usepackage[all]{xy}
\newtheorem{theorem}{Theorem}

\newtheorem{proposition}[theorem]{Proposition}
\newtheorem{corollary}[theorem]{Corollary}
\newtheorem{conjecture}[theorem]{Conjecture}
\theoremstyle{definition}
\newtheorem{definition}[theorem]{Definition}
\newtheorem{remark}[theorem]{Remark}
\newtheorem{example}[theorem]{Example}
\newtheorem{algorithm}[theorem]{Algorithm}
\newtheorem{question}[theorem]{Question}
\renewcommand{\O}{\mathcal{O}}

\newcommand{\commentaar}[1]{}
\newcommand{\C}{\ensuremath{\mathbb{C}}}
\newcommand{\R}{\ensuremath{\mathbb{R}}}
\newcommand{\Q}{\ensuremath{\mathbb{Q}}}
\newcommand{\Z}{\ensuremath{\mathbb{Z}}}

\newcommand{\F}{\ensuremath{\mathbb{F}}}
\newcommand{\A}{\ensuremath{\mathbb{A}}}

\DeclareMathOperator{\ShHom}{\mathscr{H}\text{\kern -3pt {\calligra\large om}}\,}
\renewcommand{\O}{\mathcal{O}}

\newcommand{\Spec}{\mathop\mathrm{Spec}}

\newcommand{\tors}{\mathrm{tors}}

\newcounter{nootje}
\DeclareFontFamily{U}{wncy}{}
\DeclareFontShape{U}{wncy}{m}{n}{<->wncyr10}{}
\DeclareSymbolFont{mcy}{U}{wncy}{m}{n}
\DeclareMathSymbol{\Sh}{\mathord}{mcy}{"58}

\usepackage{enumitem,calc}

\let\OLDthebibliography\thebibliography
\renewcommand\thebibliography[1]{
  \OLDthebibliography{#1}
  \setlength{\parskip}{0pt}
  \setlength{\itemsep}{0pt plus 0.3ex}
}

\newcommand{\gtwoclink}[1]{\href{https://www.lmfdb.org/Genus2Curve/Q/#1}{\textsf{#1}}}

\begin{document}
\selectlanguage{british}

\title{Efficient computation of BSD invariants in genus 2}
\author{\href{https://orcid.org/0000-0003-3767-1740}{Raymond van Bommel}\footnote{Raymond van Bommel has been supported by \href{https://simonscollab.icerm.brown.edu/}{the Simons Collaboration} on Arithmetic Geometry, Number Theory, and Computation (Simons Foundation grant 550033).}}
\cleanlookdateon
\maketitle

{\bf Abstract.} Recently, all Birch and Swinnerton-Dyer invariants, except for the order of $\Sh$, have been computed for all curves of genus 2 contained in the L-functions and Modular Forms Database \cite{LMFDB}. This  report explains the improvements made to the implementation of the algorithm described in \cite{NumericalVerification} that were needed to do the computation of the Tamagawa numbers and the real period in reasonable time. We also explain some of the more technical details of the algorithm, and give a brief overview of the methods used to compute the special value of the $L$-function and the regulator.


\newcommand{\Zp}{\mathbb{Z}_{(p)}}

\section{Introduction}

The Birch and Swinnerton-Dyer conjecture has an extensive computational history. The conjecture has originally been conceived by Birch and Swinnerton-Dyer based on computational findings on elliptic curves. Later, it has been generalised by Tate to the case of abelian varieties over general number fields. We formulate the version for principally polarised abelian varieties over $\Q$ here.

\begin{conjecture}[\textrm{BSD for abelian varieties over $\Q$, \cite{BirchSwinnertonDyer,TateBourbaki}}]
Let $A/\Q$ be a principally polarised abelian variety of dimension $d$ and algebraic rank $r$. Let $L(A,s)$ be its $L$-function, $R$ its regulator, $\Sh$ its Tate-Shafarevich group and $\Omega$ its real period. For each prime number $p$, let $c_p$ be the Tamagawa number of $A$ at $p$.
Then $L(A,s)$ has a zero of order $r$ at $s = 1$ and $$\lim_{s \rightarrow 1} {(s-1)^{-r} L(A,s)} = \frac{R \cdot \Omega \cdot |\Sh| \cdot \prod_{p} c_p}{|A(\Q)_{\tors}|^2}.$$
\end{conjecture}

Since the group $\Sh$ is very hard to compute in general, it has been common to verify the conjecture ``up to squares''. For this, one computes the other terms numerically, and checks that the value of $|\Sh|$ conjectured by the formula is close to a square or twice a square, according to what the criterion of Poonen and Stoll, \cite{PoonenStoll}, predicts.

For example, this has been done in the paper \cite{EmpiricalEvidence} by Flynn, Lepr\'evost, Schaefer, Stein, Stoll and  Wetherell for a collection of 32 modular Jacobians of hyperelliptic curves of genus 2 with small conductor. Modularity was needed for the computation of the special value of the $L$-function. In \cite{NumericalVerification}, the author developed and implemented algorithms in \texttt{Magma} to compute Tamagawa numbers and real periods for hyperelliptic curves of genus 2. The algorithms only use the equation of the curve and build upon existing methods implemented by Steve Donnelly to compute regular models.

The current goal is to compute all these BSD-invariants for the curves of genus 2 in the \cite{LMFDB}, see also \cite{Database2}, and the curves of higher genus that will appear in the future in the \cite{LMFDB}, see also \cite{Database3}. The Tamagawa numbers had already been computed for all but 54 of the genus 2 curves in the database. For the other 54 curves, and for most of the real periods, the algorithms developed in \cite{NumericalVerification} did not give a result within a reasonable amount of time.

Given that there are more than 66\,000 curves of genus 2 in the \cite{LMFDB}, the computation of a Tamagawa number or real period should ideally take at most a few seconds for most of the curves, and only more than an hour in very exceptional cases. In this report, we will explain the improvements that have been made to achieve this, and therefore finish the verification of BSD up to squares for all curves of genus 2 in the \cite{LMFDB}. Moreover, this report will also serve as a more extensive documentation of the algorithms, explaining more details of the computation than the description in \cite{NumericalVerification}.

{\bf Notation.} Throughout this text, let $C$ over $\Q$ be a smooth projective curve, let $p$ be a prime number and let $\mathcal{C} \, / \, \Zp$ be a regular model of $C$ over $\Zp$. Let $J$ be the Jacobian of $C$.

{\bf Acknowledgements.} We would like to thank Michael Stoll for sharing his code for the computation of Mordell-Weil groups, and his extensive explanation of this code. Moreover, we thank Edgar Costa for his explanation of the machinery used to compute special values of $L$-functions. Several anonymous referees are thanked for their useful comments that led to improvements of this article.

\section{Tamagawa numbers}

Classically, Tamagawa numbers are invariants which are associated to the so called Tamagawa measure on an algebraic group. However, in the context of the Birch and Swinnerton-Dyer conjecture, Tamagawa numbers are defined a bit differently. To find more about the relation between the classical definition and the one used in the Birch and Swinnerton-Dyer conjecture, see \cite{Bloch} for example. We will use the following definition.

\begin{definition}
Let $A$ be an abelian variety over $\Q$. Let $p$ be a prime, and let $\mathcal{A} / \Zp$ be a N\'eron model of $A/\Q$. Let $\varphi = \mathcal{A}_{\F_p} / \mathcal{A}_{\F_p}^0$ be the component group scheme of the special fibre of $\mathcal{A}$. Then the {\em Tamagawa number} of $A$ at $p$ is defined as $c_p(A) := | \varphi(\F_p) |$.
\end{definition}

In practice, if we have a curve of higher genus, it is already infeasible to compute equations for its Jacobian in projective space. The computation of a N\'eron model is certainly out of scope. The following theorem by Raynaud gives a relation between the component group scheme and a regular model of the curve.

\begin{theorem}[\textrm{\cite[Theorem 1.1, p.\ 3]{BoschLiu}}]\label{thm:BoschLiu}
Let $C$, and $J$ be as usual. Let $\mathcal{C}^{\textrm{hens}}$ be a regular model of $C$ over the strict henselisation of $\Zp$, and let $I$ be the set of components in the special fibre of $\mathcal{C}^{\textrm{hens}}$.
Let\vspace{-0.5cm}
\begin{itemize}
\item $\overline{\alpha}$ be the linear map $\Z^I \to \Z^I$ given by
$$\Gamma \mapsto \sum_{\Gamma' \in I} \langle \Gamma, \Gamma'\rangle \cdot \Gamma', \qquad \textrm{for } \Gamma \in I,$$
where $\langle \cdot, \cdot \rangle$ is the intersection pairing on $\mathcal{C}^{\textrm{hens}}$,
\item $\overline{\beta}$ be the linear map $\Z^I \to \Z$ given by
$$\Gamma \mapsto \mathrm{multiplicity}(\Gamma), \qquad \textrm{for } \Gamma \in I,$$
\item $\varphi_J$ be the component group of the special fibre of a N\'eron model of $J$.
\end{itemize}\vspace{-0.5cm}
Then $\mathrm{Gal}\!\left(\overline{\F_p} / \F_p\right)$ acts in a natural way on $I$ and $\Z^I$, and hence on $\mathop\mathrm{im}{\overline{\alpha}}$ and $\ker \overline{\beta}$, and there is a canonical exact sequence of $\mathrm{Gal}\!\left(\overline{\F_p} / \F_p\right)$-modules
$$0 \to \mathop\mathrm{im} \overline{\alpha} \to \ker \overline{\beta} \to \varphi_J\!\left(\,\overline{\F_p}\,\right) \to 0.$$
\end{theorem} 

Briefly summarised, our algorithm to compute the Tamagawa numbers (see also \cite{NumericalVerification}) consists of the following steps:\\[-0.8cm]
\begin{enumerate}\itemsep1pt
\item Compute a regular model $\mathcal{C}$ for $C$.
\item Find the set $I$ of components of the special fibre of $\mathcal{C}^{\mathrm{hens}}$.
\item Compute the action of Frobenius on $I$.
\item Use Theorem \ref{thm:BoschLiu} to construct $\varphi_J(\overline{\F_p})$ as a $\mathrm{Gal}\!\left(\overline{\F_p} / \F_p\right)$-module.
\item The Tamagawa number $c_p$ is then the cardinality of
$$\varphi_J(\F_p) = \varphi_J(\overline{\F_p})^{\mathrm{Gal}(\overline{\F_p} / \F_p)}.$$
\end{enumerate}

For step 1, we use the method \verb+RegularModel+ in \texttt{Magma}, which gives us patching data for a regular model. In subsection \ref{sect:regular-magma}, we explain in more detail how this model is computed and represented.

The rest of this section will be focused on step 3, the computation of the action of Frobenius on the components. Subsection \ref{sect:computation-galois-action} contains an outline of the computation of this action. The further subsections explain some problems we encountered during this computation and how we solved these.

\subsection{Computation of a regular model in \texttt{Magma}}
\label{sect:regular-magma}

The existence and construction of a regular model for $C$ relies on the resolution of singularities for arithmetic surfaces. The following result by Lipman implies that such resolutions exist.

\begin{theorem}[\textrm{\cite{Lipman, ArtinLipman}}]
Suppose $\mathcal{A}/\Zp$ is a normal proper flat model of $C/\Q$, i.e.\ $\mathcal{A}$ is an arithmetic surface over $\Zp$ and its generic fibre is isomorphic to $C$. Define a sequence 
$$\mathcal{A} = \mathcal{A}_0 \leftarrow \mathcal{A}_1 \leftarrow \mathcal{A}_2 \leftarrow \ldots$$
of arithmetic surfaces as follows: the morphism $\mathcal{A}_n \leftarrow \mathcal{A}_{n+1}$ is the normalisation of the blow-up of $\mathcal{A}_n$ in the non-regular points of $\mathcal{A}_n$. Then for sufficiently large $n$, the arithmetic surface $\mathcal{A}_n$ is regular.
\end{theorem}

In other words, we can obtain a regular model from any flat model by repeatedly normalising it and blowing up the non-regular points. This is not exactly the way in which \texttt{Magma} computes a regular model, but it is close. What the algorithm in \texttt{Magma} actually does, is constructing a sequence of blow-ups
\begin{equation}\label{eqn:magma-blowups}
\mathcal{C}_0 \leftarrow \mathcal{C}_1 \leftarrow \mathcal{C}_2 \leftarrow \ldots \leftarrow \mathcal{C}_N = \mathcal{C},
\end{equation}
where $\mathcal{C}_0$ is some proper flat model of $C$ we start with, $\mathcal{C}$ is a regular model we end with, and each morphism $\mathcal{C}_n \leftarrow \mathcal{C}_{n+1}$ is a blow-up in either a point or irreducible component of the special fibre of $\mathcal{C}_n$. If there is a choice between blowing up a point or a component, components will be blown up first (although the latter has only been implemented for curves over $\Q$).

Let us briefly describe the data involved in this production of a regular model by \texttt{Magma}. In general, the input of the algorithm is a smooth projective curve defined over a number field $K$ and a prime ideal $\mathfrak{p}$ of $\mathcal{O}_K$, the prime at which the model must become regular. The regular model is stored as a  collection of affine patches and transition maps between some of these patches. Each time a blow-up is done inside a certain patch, some new patches with transition maps to and from the original patch are created. Other relevant data that are stored are sets of non-regular points, points where components in the special fibre of the model intersect, multiplicities of these components, and intersection numbers.

In some cases, because the non-regular points in the special fibre are blown up one at a time, it could happen that such a non-regular point is not defined over $\mathcal{O}_K / \mathfrak{p}$. In this case, the number field $K$ is extended and the algorithm is restarted with this new number field as base field. Also in the end, when we want to compute the set $I$ of geometric components of the regular model (i.e.\ the set of components in the special fibre of $\mathcal{C}^{\textrm{hens}}$), the number field may need to be extended.

The reader might wonder at this point why we chose to work over number fields instead of over local fields. Indeed, the theory is cleaner to state in the local setting, but the problem is that some of the computational tools that we need have not been implemented for local fields. For example, multivariate polynomial factorisation has only been implemented for finite fields, number fields, function fields of curves over such fields, $\Z$, and some finitely generated extensions of these rings.

\subsection{Outline of algorithm to compute the Galois action}
\label{sect:computation-galois-action}

Even though our curve $C$ is defined over $\Q$, we have seen in the previous subsection that it might be necessary to work over a general number field $K$. We will go through the algorithm that \texttt{Magma} uses to create its regular models, and see how we can track the action of Frobenius on the special fibre during this process. Recall the notation from Equation \eqref{eqn:magma-blowups}.

At the start, we have a proper flat model $\mathcal{C}_0$ all of whose equations are defined using coefficients in $\Zp$. 
First we compute the irreducible components occurring in the special fibre of $\mathcal{C}_0$ and the action of Frobenius on them. This can simply be done by applying Frobenius on the ideal defining such a component in one of the affine patches containing the component. Then we proceed to look at the non-regular points.

Suppose $P$ is a non-regular point in the special fibre of $\mathcal{C}_0$. Then all Galois conjugates of $P$ are also non-regular in $\mathcal{C}_0$. Suppose $P$ is defined over $\F_{p^{\ell}}$ for some $\ell \geq 1$. Then the special fibre of the blow up of $\mathcal{C}_0$ in $P$, can be defined over $\F_{p^{\ell}}$.

Let $\mathrm{Comp}_P$ be the set of components in the special fibre of $\mathcal{C}$ that map to $P \in \mathcal{C}_0$, and define $\mathrm{Comp}_Q$ analogously for any Galois conjugate $Q$ of $P$. Then Frobenius sends components in $\mathrm{Comp}_P$ to components in $\mathrm{Comp}_{\mathrm{Frob}(P)}$. Of course we have the following properties.

\begin{proposition}\mbox{}\\[-1.1cm]
\begin{itemize} \itemsep0pt
\item[(a)] The bijection $\mathrm{Comp}_P \to \mathrm{Comp}_{\mathrm{Frob}(P)}$ induced by Frobenius preserves intersection numbers.
\item[(b)] If $Q \neq P$ is a Galois conjugate of $P$, then components in $\mathrm{Comp}_P$ and $\mathrm{Comp}_Q$ do not intersect.
\item[(c)] For any component $D$ of the special fibre of $\mathcal{C}$ and any component $E$ in $\mathrm{Comp}_{\mathrm{Frob}(P)}$ we have $$\langle D, E \rangle = \langle \mathrm{Frob}^{-1}(D), \mathrm{Frob}^{-1}(E) \rangle.$$ 
\end{itemize}
\end{proposition}

As a consequence, intersection numbers of components in $\mathrm{Comp}_{\mathrm{Frob}(P)}$ with any other component of the special fibre of $\mathcal{C}$ can be reduced to an intersection number of a component in $\mathrm{Comp}_P$ with either another component of $\mathrm{Comp}_P$, or a component not lying in any of the $\mathrm{Comp}_Q$, with $Q$ ranging over the Galois conjugates of $P$.

In particular, it is not of importance for us to compute which component of $\mathrm{Comp}_P$ is exactly mapped to which component of $\mathrm{Comp}_{\mathrm{Frob}(P)}$. Instead, our algorithm just creates $\ell$ copies of $\mathrm{Comp}_P$, and lets Frobenius map a component to the corresponding component in the next copy, for the first $\ell -1$ copies. For the $\ell$-th copy the action of Frobenius is computed differently. For this, we compute the action of $\mathrm{Frob}^{\ell}$ on $\mathrm{Comp}_P$ in a recursive way.

The algorithm backtracks the construction of the regular model, while keeping track of all the intersection numbers and the action of Frobenius. To summarise our recursive algorithm:

\begin{algorithm}\label{algo}
{\em Input: a non-regular point $P$ of $\mathcal{C}_n$. Suppose $k(P) = \F_{p^\ell}$.

\vspace{-0.4cm} Output: combinatorial description of  $\mathrm{Comp}_P$, the set of components in the special fibre of $\mathcal{C}$ that map to $P$. This combinatorial description includes the action of $\mathrm{Frob}^{\ell}$ and the intersection numbers.}

\vspace{-0.5cm}
\begin{itemize}[labelwidth=\widthof{{\bf Step 2}},leftmargin=!]
\item[{\bf Step 1}] Find equations for the components in the special fibre of $\mathcal{C}_{n+1}$ which are contracted to $P$, and compute the action of $\mathrm{Frob}^{\ell}$ on these components, and the intersections of these components with each other.
\item[{\bf Step 2}] Loop through all Galois orbits of non-regular points in $\mathcal{C}_{n+1}$ mapping to $P$.  Execute steps 3 to 7 for each orbit.
\end{itemize}\mbox{}\\[-0.9cm]
{\em Notation for steps 3 to 7: $Q$ is a point of such an orbit and $k(Q) = \mathbb{F}_{p^{\ell \cdot m}}$.}
\begin{itemize}[labelwidth=\widthof{{\bf Step 5}},leftmargin=!]
\item[{\bf Step 3}] Recursively, call Algorithm \ref{algo} with the non-regular point $Q$ of $\mathcal{C}_{n+1}$ as input. In this way, we compute the action of $\mathrm{Frob}^{\ell \cdot m}$ on $\mathrm{Comp}_Q$, and the intersections of these components with each other.
\item[{\bf Step 4}] Compute the intersection of the components considered in step 1 with the components in $\mathrm{Comp}_Q$. (This is already done by the \verb+RegularModel+ command in \texttt{Magma}.)
\item[{\bf Step 5}] Create $m-1$ more copies of $\mathrm{Comp}_Q$ for $\mathrm{Comp}_{\mathrm{Frob}^{\ell}(Q)}$, $\mathrm{Comp}_{\mathrm{Frob}^{2\ell}(Q)}$, et cetera.
\item[{\bf Step 6}] The action of $\mathrm{Frob}^{\ell}$ is the identity $\mathrm{Comp}_{\mathrm{Frob}^{\ell \cdot i}(Q)} \to \mathrm{Comp}_{\mathrm{Frob}^{\ell \cdot (i+1)}(Q)}$ for $i = 0, \ldots, m-2$. For $i = m-1$, the action is given by the action computed in step 3.
\item[{\bf Step 7}] The intersection of components contained in $\mathrm{Comp}_{\mathrm{Frob}^{\ell \cdot i}(Q)}$ and components contained in $\mathrm{Comp}_{\mathrm{Frob}^{\ell \cdot j}(Q)}$ is 0 for distinct $i,j \in \{0, \ldots, m-1\}$. The intersection number of a component $D \in \mathrm{Comp}_{\mathrm{Frob}^{\ell \cdot i}(Q)}$ for $i = 1, \ldots, m-1$ with a component $E$ considered in step 1, equals the intersection number of $\mathrm{Frob}^{-i \cdot \ell}(D)$ and $\mathrm{Frob}^{-i \cdot \ell}(E)$.
\item[{\bf Step 8}] Combine and output all the combinatorial data that has been collected.
\end{itemize}
\end{algorithm}

The improvement in comparison with the previous implementation is the more systematic implementation of Step 5, 6 and 7, which led to a significant speed-up.

\subsection{Problem arising during the algorithm}

One problem that arose regularly in our algorithm is the following. When we blow-up our point $P$ with coordinates in $\F_{p^{\ell}}$, the special fibre of the new patches are not necessarily defined over $\mathbb{F}_{p^{\ell}}$ even though they could have been. This is due to the choice we have when parametrising the blow-up and is illustrated in the following example.

\begin{example}
Suppose $K = \Q[a] / (a^4 + a + 1)$ and $\mathfrak{p} = 2\O_K$. Consider the affine curve given by $y^2 = 2(x^2 + x + 1)$ in $\A^2$ over $\O_{K, \mathfrak{p}}$. It has two Galois conjugate non-regular points $(\bar{a}^5,0)$ and $(\bar{a}^{10},0)$ in the special fibre, with coordinates in $\F_4$. Let us do a blow-up in $(\bar{a}^5, 0)$. We rewrite the equation as $$y^2 = 2 x'^{\,2} + (4a^5+2)  x' + (4a^3 - 4a),$$
where $x'=  x-a^5$. To get one of the charts of the blow-up, we `substitute' $y = tx'$ and $2 = sx'$. We get the equation
$$t^2 x'^{\,2} =  s^2x'^{\,3} + (2a^5 + 1) sx'^{\,2} + (a^3 - a)s^2 x'^{\,2}$$ 
which after dividing by $x'^{\,2}$ gives the following equations for the blow-up:
$$t^2 = s^2 x' + (2a^5 + 1)s + (a^3 - a)s^2, \qquad sx' = 2.$$
Now we see that the special fibre is not defined over $\F_4$, as $\bar{a}^3 - \bar{a} \notin \F_4$.
\end{example} 

As this problem occurred quite often, we tried to look for a solution that is efficient to implement. We are aware of two suitable possible solutions:

\vspace{-0.6cm}
\begin{enumerate}
\item[{\bf 1.}] Blow-up Galois conjugated points at the same time (e.g.\ $(\bar{a}^5, 0)$ and $(\bar{a}^{10},0)$ in the previous example). This ensures all schemes stay defined over $\Zp$. We choose not to take this route as the blow-up becomes significantly more complicated to represent when blowing up multiple points at the same time.
\item[{\bf 2.}] For any number field $K$ that we encounter in our algorithm (e.g.\ when we start with a curve over $\Q$, but we need to extend the field we are working over several times in order to blow up non-regular points, as described in section \ref{sect:regular-magma}), make sure that for each $d$ dividing $[K : \Q]$ there is a subfield $K_d \subset K$ of absolute degree $d$. In this way, if a non-regular point in the special fibre is defined over $\F_{p^d}$, all equations for the blow-up can be taken to lie inside $K_d$ and the special fibre of the blow-up is guaranteed to be defined over $\F_{p^d}$.
\end{enumerate}

In the next section, we explain what should be kept in mind when constructing number fields as in {\bf 2.}, and how it has been implemented in our algorithm.

\subsection{Construction of a suitable extension}
\label{subsect:suitable-extension}

For this subsection we say that a number field $K$ has the {\em subfield property}, if for every divisor $d$ of $[K : \Q]$ there is a subfield of $K$ of absolute degree $d$. Our goal is now to construct number fields having the subfield property, as described in the previous subsection. Let us first prove that they exist, and that we can extend them.

\begin{proposition}\label{prop:subfields}
Let $\ell$ be a positive integer and let $p$ be a prime. Let $K$ be a number field having the subfield property such that $p$ is inert in $K$. Then there exists an extension $K \subset L$ of degree $\ell$ such that $p$ is inert in $L$, and $L$ has the subfield property.
\end{proposition}

\begin{proof}
It suffices to consider the case $\ell$ is prime. Write $[K : \Q] = \ell^e \cdot m$, where $m$ is relatively prime to $\ell$. Let $K_{\ell^e}$ be a subfield of $K$ of absolute degree $\ell^e$, and let $K_m$ be a subfield of absolute degree $m$. Then construct any extension $K_{\ell^e} \subset L_{\ell^{e+1}}$ such that $p$ is inert in $L_{\ell^{e+1}}$, for example by constructing an equation for the extension of residue fields  $\F_{p^{\ell^e}} \subset \F_{p^{\ell^{e+1}}}$ and lifting it to $K_{\ell^e}$. Define $L$ as a compositum of the linearly disjoint fields $K_m$ and $L_{\ell^{e+1}}$. Then $p$ is inert in $L$. Now we show that $L$ has the subfield property.

If $d$ is a divisor of $[L : \Q] = \ell^{e+1} \cdot m$, then either $d$ divides $[K : \Q]$, or it is of the form $\ell^{e+1} \cdot m'$ for a certain divisor $m'$ of $m$. In the former case, $K$ already has a subfield of degree $d$. In the latter case, $K$ has a subfield $K_{m'}$ of degree $m'$ and the subfield $L_{\ell^{e+1}}K_{m'}$ of $L$ has degree $d$.
\end{proof}

There are of course many ways to construct fields with the subfield property, but one thing to keep in mind, is that we would like to avoid the situation where the coefficients for equations for the number field get too large to do any meaningful computation with them.

Our first attempt was to only consider abelian extensions, i.e.\ subextensions of cyclotomic fields. They typically have very small defining equations. However, the problem is that we could not answer the following question affirmatively.

\begin{question}
Is Proposition \ref{prop:subfields} still true if $K$ and $L$ are required to be abelian?
\end{question}

So we decided to construct the fields following the strategy of the proof of \ref{prop:subfields}. To keep the defining equations for our number fields small, we apply a reduction algorithm every step. This has been implemented in \texttt{Magma} under the name \verb+OptimisedRepresentation+. For this function, it is important that the ring of integers of $K$ is computable, which is a property we would like to have anyway for the computation of a regular model. A big bottleneck in the computation of the ring of integers is the factorisation of the discriminant of $K$, therefore we would like to keep this small.

Hence, we use the following strategy for each extension $K \subset L$ we construct. We first fix the extension modulo $p$, as we want $p$ to remain inert in $L$, i.e.\ we extend $K$ taking the root of some yet to be determined polynomial $f \in K[x]$ and we fix $f$ modulo $p$. We then try to find a local minimum for the discriminant of $f$. I.e., we start with any such $f$, and we repeatedly add or subtract $p$ times a randomly chosen monomial. If the discriminant of $f$ gets larger, then we revert the last change.

In practice, this worked for all curves of genus 2 contained in \cite{LMFDB}. We identified two ways in which this algorithm could still be further improved. First, there might be better ways to find `small' field extensions. Second, it might be worthwhile to determine in advance for which $d$ the subfield $K_d$ of absolute degree $d$ is actually needed. Then there would be fewer constraints and hence more candidates for the number field $K$.

\subsection{Implementation}

The algorithm has been implemented by the author in \texttt{Magma}, changing partly the way regular models are constructed. The computation of the Galois action on the component group took a negligible amount of time after a suitable regular model was constructed. Most time was spent on the construction of a regular model.

For the 66\,158 curves of genus 2 in the \cite{LMFDB}, the construction of suitable regular models for all of the primes of bad reduction, took about 2.02 seconds on average per curve. For 36 of these curves, the computation took longer than 60 seconds, the longest one taking about 1600 seconds.

\section{Real periods}

Let $A$ be an abelian variety over $\Q$ of dimension $g$. The real period has been defined in different ways in the past. Traditionally, this has been defined in terms of an integral of a $g$-form, the so-called N\'eron differential, along $A(\R)$. It can also be defined in terms of integrals of 1-forms along homology cycles of $A$. For the comparison of the different definitions see for example \cite{Gross}.

\begin{definition}
Let $m$ be the number of connected components of $A(\R)$. Let $\mathcal{A}$ be a N\'eron model of $A$ over $\Z$, and let $(\gamma_1, \ldots, \gamma_{g})$ be a basis for the group $H_1(A(\C), \Z)^{\mathrm{Gal}(\C/\R)}$ of homology cycles invariant under complex conjugation. Let $(\omega_1, \ldots, \omega_g)$ be a $\Z$-basis of $\Omega^1_{\mathcal{A}/\Z}(\mathcal{A})$. Then the real period of $A$ is defined as
$$\Omega_A = m \cdot \left| \det \left( \int_{\gamma_i} \omega_j \right)_{i,j=1}^g \right|.$$
\end{definition}

In the case of a Jacobian with N\'eron model $\mathcal{J}$, the Abel-Jacobi map gives a bijection between $\Omega^1_{\mathcal{J}/\Z}(\mathcal{J})$ and the global sections $\omega_{\mathcal{C}/\Z}(\mathcal{C})$ of the canonical sheaf of a regular model, see for example \cite[sect.\ 3.2]{NumericalVerification} and \cite[sect.\ 2]{MilneJacobian}. As a consequence, we get the following result, which gives us a practical way to compute the real period of the Jacobian of a curve.

\begin{proposition}
Let $C$, $J$ and $\mathcal{C}$ be as usual. Then 
$$\Omega_J = m \cdot \left| \det \left( \int_{\gamma_i} \omega_j \right)_{i,j=1}^g \right|,$$
where $(\gamma_1, \ldots, \gamma_g)$ is a $\Z$-basis of $H_1(C(\C), \Z)^{\mathrm{Gal}(\C/\R)}$, and $(\omega_1, \ldots, \omega_g)$ is a $\Z$-basis of $\omega_{\mathcal{C}/\Z}(\mathcal{C})$, and $m$ is the number of connected components of $J(\R)$.
\end{proposition}

An algorithm to compute real periods is described in \cite[Algorithm 13]{NumericalVerification}. We repeat it here.
\begin{algorithm}\label{algo:real-period}
{\em Input: a curve $C$ of genus $g$ over $\Q$.

\vspace{-0.4cm} Output: the real period $\Omega_J$ of its Jacobian $J$.}

\vspace{-0.5cm}\begin{itemize}[labelwidth=\widthof{{\bf Step 2}},leftmargin=!]
\item[{\bf Step 1}] Compute the big period matrix $(\int_{\gamma_i} \omega_j)_{i=1, \ldots, 2g}^{j=1, \ldots g}$. Here $\underline{\omega} = (\omega_1, \ldots, \omega_{g})$ is any basis of $\Omega^1_{C/\Q}(C)$ and $(\gamma_1, \ldots, \gamma_{2g})$ is a symplectic basis of $H_1(C(\C), \Z)$.
\item[{\bf Step 2}] For each subset $I \subset \{1, \ldots, 2g\}$ with $|I| = g$, calculate the covolume $P_I := \left|\det\left(\int_{\gamma_i} \omega_j + \overline{\int_{\gamma_i} \omega_j}\right)_{i \in I}^{j=1,\ldots, g}\right|$.
\item[{\bf Step 3}] Compute a generator $P$ for the lattice inside $\R$ spanned by the $P_I$.
\item[{\bf Step 4}] For each bad prime $p$, construct a regular model $\mathcal{C}^p / \Z_{(p)}$ of $C$.
\item[{\bf Step 5}] For each bad prime $p$ and each of the differentials $\omega_1, \ldots, \omega_g$, check if the differential has a pole on any of the irreducible components of the special fibre of $\mathcal{C}^p$. If so, adjust the basis by multiplying the differential having a pole with $p$ to get a new basis $\underline{\omega}'$ and apply Step 5 again (until the basis is not changing anymore).
\item[{\bf Step 6}] For each bad prime $p$ and each $(c_j)_{j=1}^g \in \{0, \ldots, p-1\}^g \setminus \{(0,0, \ldots, 0)\}$, check if $\sum_j c_j \omega_j$ vanishes on the whole special fibre of $\mathcal{C}^p$. If so, adjust the basis $\underline{\omega}'$ by replacing one of the $\omega_j$ such that $c_j \neq 0$ with $\tfrac1p \sum_j c_j \omega_j$, then apply Step 6 again (until the basis is not changing anymore).
\item[{\bf Step 7}] For each bad prime $p$ compute $p^{a-b}$, where $a$ is the number of basis adjustments done in Step 5, and $b$ is the number of basis adjustments done in Step 6 (this is also the determinant of the change of basis matrix whose columns express $\underline{\omega}'$ in terms of $\underline{\omega}$). Then take the product $W$ over $p$ of these determinants, and output $W \cdot P$.
\end{itemize}
\end{algorithm}

In order to compute the real periods for the 66\,158  hyperelliptic curves of genus 2 contained in \cite{LMFDB} some optimisations had to be done:
\begin{enumerate}
\item In our original implementation, many computations appeared to be done multiple times. To avoid this, new data structures have been constructed to store all rings, ideals, their Gr\"obner bases and other relevant objects that are needed multiple times during the computation. 

\item In our original implementation, we used Van Wamelen's algorithm \cite{VanWamelen} implemented in \texttt{Magma} to compute the big period matrix in Step 1. This only works for hyperelliptic curves with a simplified Weierstra\ss\ model of odd degree. We moved to Pascal Molin and Christian Neurohr's implementation that can compute the Riemann surface associated to any curve (\cite{MolinNeurohr,Neurohr}). This also had advantages for the hyperelliptic case as it allowed us to take simpler (e.g.\ even degree) models, which typically give rise to simpler regular models.

\item In Step 6 of the algorithm, for $p^g - 1$ differentials it is tested whether or not they vanish on the special fibre. This should not be necessary and in principle be a linear algebra problem. This is explained in more detail in subsection \ref{subsection:linear-algebra-solution}.

\item The Gr\"obner basis computations over $\Z$ took most of the time in this algorithm. In some cases, even the computation of the function field over $\Q$ of one of the patches of the regular model was taking a lot of time. In subsection \ref{subsection:computations-over-ZpnZ}, it is explained how we circumvented Gr\"obner basis computations over $\Z$.
\end{enumerate}

Moreover, we will consider some other ideas that could be used to improve this code even further in the future.

\subsection{Computation of the order of vanishing of a function}

First, we will briefly explain how the order of vanishing of a function is computed. We reduced our computation to one inside a polynomial ring over $\Z$. In the following example, we illustrate how this is done.

\begin{example}\label{ex:multiplicity}
Consider the scheme given by the equation $y^2 = x^2 - 2x - 2$ in $\A^2_{\Z}$. We take $p = 2$. We want to compute the order of vanishing of the function 2 on the component given by the vanishing of $x+y$ and 2. In other words, we want to compute the multiplicity of that component in the special fibre.

We start with the ideal $I = (x+y,\ 2) \subset \Z[x,y]$ and notice that 2 lies in it, so the order of vanishing is at least 1. Then we compute $$I_2 = I^2 + (y^2 - x^2 + 2x + 2) = ( x^2 + y^2 + 2, \ 2x + 2, \ 2y + 2,\ 4).$$
Even though 2 does not lie in $I_2$, we see that the ideal quotient
$$(I_2 : (2)) = (x+1,\ y+1,\ 2) \not \subset I.$$
In other words: 2 when multiplied with some unit lands in $I_2$, and 2 vanishes with order at least 2. We proceed to compute
\begin{align*}
I_3 &= I^3 + (y^2 - x^2 + 2x + 2)\\
&= (x^2 + 2x + 3y^2 + 6,\ 2xy + 2x + 2y^2 + 2y,\ 4x+4,\ 4y+4,\ 8)
\end{align*}
and
$$(I_3 : (2)) = (x^2 + y^2 + 2,\ xy + x + y^2 + y,\ 2x + 2,\ 2y + 2,\ 4),$$
which is easily seen to be contained in $I$. Hence, the vanishing order of the function 2 on the component given by the vanishing of $x+y$ and $2$, equals 2.
\end{example}

So in order to compute the order of vanishing, we make use of the following proposition.
\begin{proposition}\label{prop:vanishing-computation}
Let $R$ be a commutative ring. Let $J$ be an ideal of $R$, and let $I$ be a prime ideal of $R$ containing $J$. Let $n$ be an integer. Then an element $f \in R$ maps to an element in $I^n \cdot (R/J)_I \subset (R/J)_I$, i.e.\ $f$ vanishes with order at least $n$ on the component defined by $I$ in $\Spec(R/J)$, if and only if the ideal quotient $(I^n + J : (f))$ is not contained in $I$.
\end{proposition}

\begin{proof}
Suppose $\overline{f} \in R/J$ lands in $I^n \cdot (R/J)_I$ after localisation. This means that there are $g \in I^n$ and $h \notin I$, such that $(\overline{f} - \overline{g})  \overline{h} = 0 \in R/J$, or in other words $(f - g)h \in J$. From this, we deduce that $h \in (I^n + J : (f))$, as desired.

Conversely, if $h \in (I^n + J : (f))$ is an element such that $h \notin I$, then we can write $hf$ as $j + g$ for some $j \in J$ and $g \in I^n$. Then we get $\overline{h}\overline{f} = \overline{g} \in  I^n \cdot (R/J)_I$. As $\overline{h}$ is a unit inside $(R/J)_I$ we get that $\overline{f} \in I^n \cdot (R/J)_I$, as desired. 
\end{proof}

This leads to the following algorithm.

\begin{algorithm}\label{algo:vanishing-order}{\em
Input: an $f \in \Z[x_1, \ldots, x_m]$ and ideals $I$ and $J$ of $\Z[x_1, \ldots, x_m]$ such that $V(J)$ is a flat curve over $\Z$ on which $f$ does not vanish. We assume $I$ is a prime ideal containing $J$, defining an irreducible component in the special fibre $V(J)_p$ of this curve. We assume that the local ring $\O_{V(J),I}$  is regular.

\vspace{-0.4cm} Output: the order of vanishing of $f$ at $V(I)$.}

\vspace{-0.5cm}\begin{itemize}[labelwidth=\widthof{{\bf Step 2}},leftmargin=!]
\item[{\bf Step 1}] If $f$ is not irreducible, find a factorisation for $f$ and run the rest of the algorithm for each of its irreducible factors.
\item[{\bf Step 2}] Start with $n = 1$ and $I_0 = J$.
\item[{\bf Step 3}] Compute a Gr\"obner basis for the ideal $I_n = I_{n-1} \cdot I + J$ of $\Z[x_1, \ldots, x_m]$.
\item[{\bf Step 4}] Compute generators for $(I_n : (f))$.
\item[{\bf Step 5}] If any of these generators is not contained in $I$, then return the value $n-1$. Else, increase $n$ by 1 and proceed with Step 3.
\end{itemize}
\end{algorithm}

\begin{proposition}
Algorithm \ref{algo:vanishing-order} terminates and is correct.
\end{proposition}

\begin{proof}
By induction we see that $I_n = I^n + J$. The correctness follows immediately from Proposition \ref{prop:vanishing-computation}. As $f$ does not vanish on $V(J)$, the element $f$ reduces to a non-zero element of the discrete valuation ring $\O_{V(J),I}$. In particular, the order of vanishing of $f$ at $V(I)$ is finite, and the algorithm will terminate.
\end{proof}

The ideal quotient and membership of ideals can be checked using Gr\"obner basis machinery. This has been implemented in \texttt{Magma} and other packages in the case $R$ is a polynomial ring over a field or a polynomial ring over $\Z$.

In case we are working over a ring of integers $\O_K$ which is not $\Z$, we can represent $\O_K$ as a $\Z$-algebra of finite type and still use the Gr\"obner basis machinery over $\Z$. For the computation of the real period, the field extensions that are necessary when computing a regular model, as explained in subsection \ref{sect:regular-magma}, are not chosen as in subsection \ref{subsect:suitable-extension}, but they are chosen in such a way that $\O_K$ is monogenic. In this way, we can proceed with the Gr\"obner basis calculations at the cost of adding just one extra variable. This is how we get an algorithm to compute the order of vanishing of a polynomial.

In general, a function will be of the form $\frac{f}{g}$, where $f$ and $g$ can be written as polynomials with integer coefficients. The order is then computed as the difference between the order of vanishing of $f$ and $g$.

The reader might be wondering at this point why we are working over $\Z$ instead of over the ring of integers of a local field. The main reason is the availability of an algorithm to factor multivariate polynomials over $\Z$ in \texttt{Magma}. If $f$ and $g$ are two polynomials, then the computation of the order of vanishing of $f$ and $g$ generally takes way less time than the computation of the order of vanishing of $f \cdot g$ (without using the factorisation).

\subsection{How to avoid difficult computations in characteristic 0}
\label{subsection:computations-over-ZpnZ}

\subsubsection{Working over $\Z/p^n\Z$ instead of over $\Z$}
\label{subsubsection:mod-ZpnZ}

In Example \ref{ex:multiplicity}, all the calculations could have been done over $\Z/4\Z$ instead of $\Z$, as we already know in advance that 4 has a strictly higher order of vanishing than 2. In fact, we replaced the current algorithm to compute the multiplicity of a component with this new algorithm working over $\Z/p^2\Z$, leading to big speed-ups in some difficult cases.

In general, when we consider a polynomial $f$ and we want to determine whether $f$ vanishes with order at least $r$ in a component with multiplicity $m$, it suffices to do all computations over $\Z / p^{\lfloor r/m \rfloor + 1} \Z$. In this way, we can avoid Gr\"obner basis computations over $\Z$, which take significantly more time.

\begin{proposition}
Let $f,r,m$ be as above. If Algorithm \ref{algo:vanishing-order} is run over $\Z / p^{\lfloor r/m \rfloor + 1} \Z$ instead of $\Z$ and outputs a value less than $r$, then the output is correct. If the output is greater than or equal to $r$, then $f$ vanishes with order at least $r$.
\end{proposition}

\begin{proof}
Instead of testing whether $f$ lies in $I^n \cdot (R/J)_I$, see also Proposition \ref{prop:vanishing-computation}, we are testing whether $f + p^{\lfloor r/m \rfloor + 1} \cdot g$ lies in $I^n \cdot (R/J)_I$, for some element $g \in \Z[x_1, \ldots, x_m]$. Because $p^{\lfloor r/m \rfloor + 1}$ vanishes with order $m \cdot (\lfloor r/m \rfloor + 1) > r$, the strong triangle inequality yields that $f \in I^n \cdot (R/J)_I$ if and only if $f + p^{\lfloor r/m \rfloor + 1} \cdot g \in I^n \cdot (R/J)_I$ in case $n \leq r$. This proves the proposition.
\end{proof}

\subsubsection{Optimising the ideals for fast computation}

Another big improvement came from a slight modification of the ideals $I_n$, as described in Example \ref{ex:multiplicity}. A priori we were computing $I_n$ as $I_n = I^n + J$, where $J$ is an ideal inside a polynomial ring $\Z[x_1, \ldots, x_{\ell}]$ corresponding to an affine patch of $\mathcal{C}$, and then computing a Gr\"obner basis. In our new implementation we defined ideals $I_n^{\mathrm{modified}}$ and $J_n$ as follows:
$$J_n = \begin{cases} I &\textrm{if $n = 1$,} \\ I_{n-1}^{\mathrm{modified}} \cdot I + J &\textrm{else,} \end{cases}$$
$$I_n^{\mathrm{modified}} = \begin{cases} I &\textrm{if $n = 1$,}\\ J_n  + \langle x \in \mathrm{GB}(I_{n-1}^{\mathrm{modified}}) :  (J_n : (x)) \not\subset I \rangle &\textrm{else,}\end{cases}$$
where $\mathrm{GB}$ gives a Gr\"obner basis of an ideal. Basically, the elements that we added to $I_n^{\mathrm{modified}}$ are exactly those generators of $I_{n-1}^{\mathrm{modified}}$ that already appear to vanish with order $n$ at $I$. The following proposition will formalise this and prove that we still get the correct answer if use the ideal $I_n^{\mathrm{modified}}$ instead of $I_n$.

\begin{proposition}
For any positive integer $n$, there is an equality $$(I_n/J)_{I} = (I_n^{\mathrm{modified}}/J)_{I}$$
of ideals of $(\Z[x_1, \ldots, x_{\ell}]/J)_{I}$. 
\end{proposition}

\begin{proof}
We prove the statement by induction to $n$. The case $n=1$ is trivial.

Let $k \geq 2$ be an integer and suppose the statement is true for all $n < k$. We start with the equality $(I_{k-1}/J)_{I} = (I_{k-1}^{\mathrm{modified}}/J)_I$. Multiplying both sides with the ideal $(I/J)_I$, gives us that $(I_k/J)_I = (J_k/J)_I$. Now we will prove the equality $(J_k/J)_I = (I_k^{\mathrm{modified}}/J)_I$, which will finish the proof of the proposition.

Let $x \in \Z[x_1, \ldots, x_{\ell}]$ by any element with $(J_k : (x)) \not\subset I$. This means that $x$ multiplied with a unit in $R_I$ lies inside $J_k$. In other words, $\overline{x}$ lies in $(J_k/J)_I$. Therefore, the equality $(J_k/J)_I = (I_k^{\mathrm{modified}}/J)_I$ follows by definition of $I_k^{\mathrm{modified}}$.
\end{proof}

\begin{corollary}
Algorithm \ref{algo:vanishing-order} is still correct if $I_n$ is replaced by $I_n^{\mathrm{modified}}$.
\end{corollary}


One might wonder why we chose to use this rather artificial looking ideal $I_n^{\mathrm{modified}}$ instead of $I_n$. The following example, in which we also use our improvements from the previous subsection \ref{subsubsection:mod-ZpnZ}, illustrates the advantage of this approach.
\begin{example}
Consider the ideal $$J = (y^{100} - x^{100} + 2x + 2, \ xz - 2)$$ inside the ring $\frac{\Z}{2^{18}\Z}[x,y,z]$ equipped with the graded reverse lexicographic order. Now we look at the component in the special fibre given by $I = (x+y, z, 2)$. Then the direct computation of a Gr\"obner basis for $I_{18}$ took 43.5 seconds, while the inductive computation of a Gr\"obner basis for $I_{18}^{\mathrm{modified}}$ only took 2.32 seconds. Moreover, the former Gr\"obner basis has 61 elements, while the latter only has 30 elements.
\end{example}

\subsubsection{Avoiding function field computations}

In our previous implementation, we computed a function field for each patch of the regular model. This function field was used to represent differentials on this patch. The computation of this function field is quite expensive. In the following example, it did not seem to finish within reasonable time.

\begin{example}
Consider the curve \gtwoclink{1328.a.84992.1}, given by $$y^2 + (x + 1)y = 4x^5 + 9x^4 + 16x^3 + 13x^2 + 8x + 1.$$ The regular model that \texttt{Magma} computes at 2 for this curve, has a patch given by two equations. One of these equations has 428 terms and is of degree 56. The computation of a function field over $\Q$ for this patch did not finish in 24 hours. Although we could not determine why the computation took so long, we suspect that it is due to the exponential time needed to compute a Gr\"obner basis for an ideal.
\end{example}

However, in practice, we do not need this function field. Except for the first patches, every patch $\mathcal{P}$ is defined using three variables $x$, $y$, and $z$, and two equations $f = g = 0$. Without loss of generality, we can assume that all differentials on $\mathcal{P}$ can be represented as $h \cdot dx$ with $h \in K(x, y, z)$, where $K$ is the base field (this does not need to be $\Q$, as sometimes there is need extend the base field, as explained in subsection \ref{sect:regular-magma}). The only operation for which we used the function field, was to convert a differential of the shape $j \cdot dy$ or $j \cdot dz$ into a differential of the shape $h \cdot dx$. The following algorithm gives us a way to achieve this in polynomial time.

\begin{algorithm} \label{algo:differentials} {\em
Input: polynomials $f,g \in K[x,y,z]$ defining the affine patch $\mathcal{P}$.

\vspace{-0.4cm} Output: relatively prime $a,b \in K[x,y,z]$ such that $a \cdot dx = b \cdot dy$ holds on $\mathcal{P}$.}

\vspace{-0.5cm}\begin{itemize}[labelwidth=\widthof{{\bf Step 2}},leftmargin=!]
\item[{\bf Step 1}] The equations $df = dg = 0$, give rise to the following linear relations over $K[x,y,z]$ for $dx$, $dy$ and $dz$:
\begin{equation}\label{eqn:differentials}
\begin{pmatrix}
\frac{\partial f}{\partial x}	&\frac{\partial f}{\partial y}	&\frac{\partial f}{\partial z}	\\[0.1cm]
\frac{\partial g}{\partial x}	&\frac{\partial g}{\partial y}	&\frac{\partial g}{\partial z}	\\
\end{pmatrix} \cdot \begin{pmatrix}
dx \\ dy \\ dz \end{pmatrix} = 0.
\end{equation}

\item[{\bf Step 2}] Consider $$M = \begin{pmatrix}\frac{\partial f}{\partial y}	&\frac{\partial f}{\partial z}	\\[0.1cm]
\frac{\partial g}{\partial y}	&\frac{\partial g}{\partial z}\end{pmatrix}.$$
Multiply equation \eqref{eqn:differentials} on the left side by the adjugate of $M$ to get a relation of the form
\begin{equation}\label{eqn:diff2}
\begin{pmatrix}
F	&D	&0	\\[0.1cm]
G	&0	&D	\\
\end{pmatrix} \cdot \begin{pmatrix}
dx \\ dy \\ dz \end{pmatrix} = 0,
\end{equation}
for some polynomials $D$, $F$ and $G$ in $K[x,y,z]$.

\item[{\bf Step 3}] Then we compute $a = \frac{F}{\mathrm{gcd}(D,F)}$ and $b = \frac{D}{\mathrm{gcd}(D,F)}$.
\end{itemize}
\end{algorithm}

\begin{proposition}
Algorithm \ref{algo:differentials} is correct and the number of field operations in $K$ is polynomial in the input size.
\end{proposition}

\begin{proof}
It is clear that equality \eqref{eqn:diff2} always holds. To check that Step 3 is well-defined, we will prove that $D \neq 0$. Suppose $D = 0$. Because $dx$ is a non-vanishing differential by assumption, we get $F = G = 0$. The Jacobian matrix $$N = \begin{pmatrix}
\frac{\partial f}{\partial x}	&\frac{\partial f}{\partial y}	&\frac{\partial f}{\partial z}	\\[0.1cm]
\frac{\partial g}{\partial x}	&\frac{\partial g}{\partial y}	&\frac{\partial g}{\partial z}	\\
\end{pmatrix}$$ has full rank, as $f = g = 0$ defines a curve in $\A^3$. Hence $\mathrm{adj}(M)N = 0$ implies that $\mathrm{adj}(M) = 0$ and $M = 0$. But that immediately contradicts the fact that $N$ is of full rank.

Now we will look at the number of field operations, let $B$ be a bound for the number of coefficients of $f$ and $g$. To compute the six partial derivatives, we need at most $\O(B)$ field operations. For the matrix multiplication in Step 3, we need to multiply polynomials with at most $B$ coefficients with each other, so this can be done using at most $\O(B^2)$ field operations. To see that the number of field operations needed in Step 3 is polynomial, we refer the reader to \cite[chap.\ 6]{Wittkopf}.
\end{proof}

\subsection{How to avoid checking \texorpdfstring{$p^g$}{p-to-the-g} differentials}
\label{subsection:linear-algebra-solution}

In Step 6 of Algorithm \ref{algo:real-period}, a candidate basis $\omega_1, \ldots, \omega_g$ for the global sections of the canonical sheaf $\omega_{\mathcal{C} / \Zp}$ is considered. For each linear combination $\sum_j c_j \omega_j$ with the $c_j \in \{0, \ldots, p-1\}$ not all equal to zero, it is tested whether $\sum_j c_j \omega_j$ vanishes on the special fibre. If so, then $\tfrac1p \sum_j c_j \omega_j$ can replace one of the current candidate basis elements.

For big values of $p$, the checking of $p^g-1$ differentials can take a lot of time. To check whether a differential vanishes on all components of the special fibre of $\mathcal{C}$, quite a few expensive Gr\"obner basis computations are needed. We would like to limit such computations as much as possible.  Of course, it would suffice to actually check just $\frac{p^g-1}{p-1}$ of these differentials, but for $g \geq 2$, this could still be a big number.

If $D$ is a component of the special fibre $\mathcal{C}_p$, then $\O_{\mathcal{C},D}$ is a discrete valuation ring whose residue field $\F_D$ is a function field over a finite field. So in principle, checking whether $\sum_j c_j \omega_j$ is vanishing on $D$ is an $\F_p$-linear algebra problem inside $\F_D$, which could be solved much more efficiently.

If we write $\omega_j = \frac{f_j}{h_j} d$ for polynomials $f_j, h_j$ and $d$ is a local generator on a sufficiently small affine neighbourhood of the generic point of $D$ for the sheaf $\omega_{\mathcal{C}/\Zp}$, as in \cite[Lemma~12]{NumericalVerification}, then we would need to determine the residue class of $\frac{f_j}{h_j}$ inside $\F_D$. The problem that we now encounter is that both $f_j$ and $h_j$ could vanish on $D$ with the same order $n > 0$.
We could try to check if $f_j - \alpha \cdot h_j$ vanishes on $D$ with order greater than $n$, where $\alpha$ ranges over lifts of elements in $\F_D$, but then we still need to check $|\F_D|$ elements, which is exactly what we wanted to avoid. Although the vanishing of $h_j$ on $D$ sometimes happens, it seemed that in most cases there were a bunch of components for which this does not happen. Therefore, we decided to only use a linear algebra approach on those components on which $h_j$ does not vanish. We then check in the end if the differential we found really vanishes on the whole special fibre.

\begin{remark}\label{remark:dividing-by-uniformiser}
An alternative approach would be to try to repeatedly divide a uniformiser out of $f_j$ and $h_j$ until they both have order 0. However, the difficulty with this approach is, that we are doing our computations over $\Z/p^e\Z$ for some exponent $e$ (see subsection \ref{subsubsection:mod-ZpnZ}), and we cannot divide by the nilpotent uniformiser.
\end{remark}

Instead of actually doing linear algebra in the different function fields $\F_D$, we opted for a faster approach, which works in almost all of the cases. We precompute a finite set $S$ of (not necessarily rational) random closed points in different components of the special fibre, and specialise to these points. This leads to the following algorithm.

\begin{algorithm}\label{algo:subspace} {\em
Input: precomputed set $S$ of closed points on the special fibre of $\mathcal{C}$, polynomials $f_1, \ldots, f_g, h_1, \ldots, h_g$ such that $\omega_j = \frac{f_j}{h_j} d$ for $j = 1, \ldots, g$.

\vspace{-0.4cm} Output: a subspace $V$ of $\F_p^g$ such that all $(c_1, \ldots, c_g) \in \F_p^g$ such that $\sum_j c_j \omega_j$ vanishes on the special fibre of $\mathcal{C}$ are contained in $V$.}

\vspace{-0.5cm}\begin{itemize}[labelwidth=\widthof{{\bf Step 2}},leftmargin=!]
\item[{\bf Step 1}] For each $j = 1, \ldots, g$ compute the subset $S_j = \{P \in S : h_j(P) = 0\}$. Compute $S' = S \setminus \bigcup_{j=1}^g S_j$.
\item[{\bf Step 2}] Construct the $|S'| \times g$ matrix corresponding to the linear map $\F_p^g \to \F_p^{S'}$ mapping the $j$-th standard basis vector to $( f_j(P) \cdot h_j(P)^{-1})_{P \in S'}$.
\item[{\bf Step 3}] Output the kernel of $M$.
\end{itemize}
\end{algorithm}

\begin{proposition}
Algorithm \ref{algo:subspace} is correct, i.e.\ the vector space $V$ satisfies the output conditions, and the number of field operations needed is polynomial in the input size.
\end{proposition}

\begin{proof}
If $\sum_j c_j \omega_j$ vanishes on the special fibre of $\mathcal{C}$, then $\sum_j c_j \frac{f_j}{h_j}$ must vanish on the special fibre. In particular, $\sum_j c_j \frac{f_j}{h_j}$ must vanish in any point of $S'$, and $(c_j)_{j=1}^g$ lies in the kernel of the linear map constructed in Step 2.

The evaluation of the polynomials in Step 1 is polynomial in the input size: at most $\O(g \cdot |S| \cdot B \cdot \log{D})$ operations are needed, where $B$ is a bound for the number of coefficients of $f_1, \ldots, h_g$, and $D$ is a bound for their degree. For Step 2, we need $\O(g \cdot |S'| \cdot B \cdot \log{D})$ operations to construct a matrix for the linear map. Finally, the Gauss elimination to find generators for the kernel can be done using $\O(\max(g, |S'|)^3)$ field operations.
\end{proof}

In most cases in our computation, the vector space $V$ appeared to be at most 1-dimensional and the full computation to determine whether a differential vanishes, had to be done at most once. In almost all cases for which the space had dimension greater than 1, the prime $p$ appeared to be small.

\subsection{Possible ideas for improvements}

Let $D$ be a component of the special fibre of $\mathcal{C}$. We propose two ideas that could help in the task of finding the order of vanishing of a function $f$ at $D$.

The first idea is to try to remedy the defect described in Remark \ref{remark:dividing-by-uniformiser}. For example, if we know that an integer is $4 \mod 12$ and we divide it by 2, we cannot tell its value modulo 12, but we do know for sure that it is $2 \mod 6$. In the same way, we can reduce the modulus when we divide by a uniformiser to solve the problem described in Remark \ref{remark:dividing-by-uniformiser}.

The second idea comes from a comparison with methods used in real analysis. If you want to see if a $C^{\infty}$ function which vanishes in a point $P$, is vanishing twice, you usually compute its derivative and check if that is vanishing. Of course, this does not work in characteristic $p$, but the Hasse derivative could work, see for example \cite[sect.\ 1.3]{Goldschmidt}. If there are methods to compute the Hasse derivative more efficiently than the Gr\"obner basis computations that we are currently doing, this could lead to much smaller runtimes.

\subsection{Implementation}

The algorithm has been implemented by the author in \texttt{Magma} and used to compute the real period for all 66\,158 curves of genus 2 in the \cite{LMFDB}, all hyperelliptic curves and all but 8 non-hyperelliptic curves of genus 3 in \cite{Database3}.

The average runtime was about 1.67 seconds per curve for the 66\,158 curves of genus 2. For 89 of these curves the computation took more than 60 seconds, the longest runtime for a single curve being 6668 seconds. For the first 200 curves appearing in the database, and for the slowest curve, we have more detailed information showing how the runtime has been spent between the different parts of the computation:

\begin{figure}[ht]
\begin{tabular}{|l|c|c|}\hline
{\em Average time per curve (in seconds)}	&first 200 curves	&slowest curve \\ \hline
Computation of big period matrix	&1.20	&1.31 \\ \hline
Construction of regular models		&0.22	&1.22 \\ \hline
Finding a basis of $\omega_{\mathcal{C}/\Zp}(\mathcal{C})$	&1.03	&6665.40 \\ \hline
{\bf Total runtime}	&2.46	&6668.03 \\ \hline
\end{tabular}
\caption{Average runtime per curve}
\end{figure}

Almost all of the time spent finding a basis of $\omega_{\mathcal{C}/\Zp}(\mathcal{C})$ was spent on actual Gr\"obner basis computations.

\section{Other BSD invariants}

There have also been recent computations of other Birch and Swinnerton-Dyer invariants of Jacobians of genus 2 curves to the \cite{LMFDB}. In this section, we summarise them.

The leading coefficients for the $L$-functions of these Jacobians have been computed using methods developed by Bober, Booker, Costa, Lee, Platt and Sutherland, to appear in \cite{motiviclfunc}, which builds upon work of Booker, see \cite{Booker}. The code developed by Costa and Platt is available at \cite{lfunccode}.

Generators for the Mordell-Weil group of the Jacobians have been computed using code of Stoll. This is an improvement of the old \texttt{j-points} code. The algorithm roughly consists of the following three phases:

\vspace{-0.7cm}\begin{itemize}\itemsep0pt
\item[1.] try to find a tight upper bound $R$ for rank of the Mordell-Weil group, i.e.\ we want $R$ to equal the rank of $J$, but at this point we can only prove that the rank of $J$ is bounded by $R$;
\item[2.] search for points in $J(\Q)$, until they generate a subgroup of rank $R$;
\item[3.] saturate this subgroup to obtain the the Mordell-Weil group.
\end{itemize}

For the upper bound for the rank, the following methods are used:

\vspace{-0.7cm}
\begin{itemize}\itemsep0pt
\item[-] the 2-Selmer group of the Jacobian and the 2-Selmer set of the $\mathrm{Pic}^1$ of the curve can be computed and used to obtain bounds for the rank, see \cite{Stoll2descent,Creutz},
\item[-] the method of visualisation of elements of $\Sh$ can be used to further improve these bounds, see \cite{CremonaMazur,BruinVisualisation,BruinFlynn},
\item[-] 2-power isogenous abelian surfaces are computed and the same methods are used to compute upper bounds for the rank of their Mordell-Weil group, which turned out to be better than the upper bound found for $J$ itself in some cases.
\end{itemize}

For the search phase, the strategy is to look for points up to a certain height. As the canonical height is hard to control, a naive height function is used in this phase. The difference between the naive and canonical height is bounded by some positive real number $\delta$, that we can actually compute. So to find all points up to canonical height $B$, one has to enumerate all points up to naive height $B + \delta$. In the new version of \texttt{j-points}, a new modified naive height function has been used, which has a smaller difference with the canonical height than the classical naive height function, see \cite[part IV]{MuellerStoll}. This reduces the size of the required search space significantly.

For the saturation phase, it is important that we have found all points up to canonical height $\varepsilon$ for some $\varepsilon > 0$. If all generators for the free part of the subgroup we found have canonical height at most $H$, then the index of the subgroup inside the Mordell-Weil group can only have prime factors that are smaller than $\frac{H}{\varepsilon}$. We can then saturate the subgroup at every prime number $p < \frac{H}{\varepsilon}$ and in this way, we are sure that we found the whole Mordell-Weil group. The saturation has also been made faster using the modified naive height function, see also \cite[part IV]{MuellerStoll}.

With this code, we have been able to compute generators for almost all Mordell-Weil groups of genus 2 curves in the \cite{LMFDB}. For the few for which we did not directly find generators, we found generators for an abelian surface isogenous to the Jacobian. The following example is the one for which it was most difficult to find a generator.
\begin{example}
The Jacobian for the genus 2 curve \gtwoclink{900617.a.900617.1} given by
$$y^2 + (x^2 + x)y = x^5 - 65x^4 + 224x^3 + 30x^2 + x$$
is isogenous to the Jacobian of the genus 2 curve given by
$$y^2 = -1110x^6 - 2790x^5 - 9315x^4 + 11160x^3 + 18315x^2 - 9540x - 12765.$$
On the latter Jacobian, we were able to find a generator of the rank 1 Mordell-Weil group of canonical height $16.281246$ using the code of Stoll. We then computed the image of this point under the isogeny, and used the modified height bound and improved saturation, to prove that the we actually found a generator of canonical height $65.124982$ on the Jacobian of \gtwoclink{900617.a.900617.1}.
\end{example}

For the computation of the regulators from the set of generators of the Mordell-Weil group, we used an algorithm which is based on an approach to decompose the height in different local contributions using Arakelov theory, see also \cite{Holmesheights,Muellerheights,Nonhyperellipticheights}.

\end{document}